\documentclass[11pt]{amsart}

\usepackage{paralist, amsmath, amsthm, amssymb, color, graphicx, hyperref}
\usepackage[margin=1.4in]{geometry}
\DeclareGraphicsExtensions{.jpg,.pdf,.eps}
\graphicspath{{./Figures/}{./}}

\newtheorem{thm}{Theorem}
\numberwithin{thm}{section}

\newtheorem{prop}[thm]{Proposition}
\newtheorem{lemma}[thm]{Lemma}
\newtheorem{cor}[thm]{Corollary}

\theoremstyle{definition}
\newtheorem{definition}[thm]{Definition}
\newtheorem{notation}[thm]{Notation}

\theoremstyle{remark}

\newtheorem{remark}[thm]{Remark}

\newtheorem{example}[thm]{Example}

\newcommand{\ob}[1]{#1}

\newcommand{\he}{\mathbf{h}}
\newcommand{\she}{{|\he|}}

\newcommand{\ga}{\gamma}
\newcommand{\de}{\delta}
\newcommand{\eps}{\epsilon}

\newcommand{\la}{\lambda}
\newcommand{\om}{\omega}

\newcommand{\QQ}{\mathbb{Q}}

\newcommand{\ZZ}{\mathbb{Z}}
\newcommand{\NN}{\mathbb{N}}
\newcommand{\PP}{\mathbb{P}}

\newcommand{\mA}{\mathcal{A}}
\newcommand{\mB}{\mathcal{B}}
\newcommand{\mC}{\mathcal{C}}
\newcommand{\mD}{\mathcal{D}}
\newcommand{\mE}{\mathcal{E}}

\newcommand{\mH}{\mathcal{H}}

\newcommand{\mP}{\mathcal{P}}

\newcommand{\mT}{\mathcal{T}}
\newcommand{\mS}{\mathcal{S}}

\newcommand{\mm}{\mathbf{m}}

\newcommand{\xx}{\mathbf{x}}
\newcommand{\yy}{\mathbf{y}}
\newcommand{\zz}{\mathbf{z}}

\newcommand{\Sym}{\mathop{Sym}}
\newcommand{\1}{\mathbf{1}}

\newcommand{\Pd}{{\widehat{\chi}_{d}}}
\newcommand{\Pone}{{\widehat{\chi}_{1}}}
\newcommand{\Ptwo}{{\widehat{\chi}_{2}}}
\newcommand{\Pzero}{{\widehat{\chi}_{0}}}

\newcommand{\ov}{\overline}
\newcommand{\ds}{\displaystyle}
\newcommand{\fig}[3]{\begin{figure}[!ht]\begin{center}\includegraphics[#1]{#2}\end{center}\caption{#3}\label{fig:#2}\end{figure}}

\title[Combinatorial reciprocity for the chromatic polynomial]{Combinatorial reciprocity for the chromatic polynomial and the chromatic symmetric function}

\author{Olivier Bernardi and Philippe Nadeau}
\thanks{O.B. is partially supported by NSF grant DMS-1800681.}
\date{\today}

\begin{document}

\begin{abstract}
Let $G$ be a graph, and let $\chi_G$ be its chromatic polynomial.
For any non-negative integers $i,j$, we give an interpretation for the evaluation $\chi_G^{(i)}(-j)$ in terms of acyclic orientations. This
 recovers the classical interpretations due to Stanley and to Greene and Zaslavsky respectively in the cases $i=0$ and $j=0$. We also give symmetric function refinements of our interpretations, and some extensions. The proofs use heap theory in the spirit of a 1999 paper of Gessel.
\end{abstract}

\maketitle

\section{Introduction}
Let $G$ be a (finite, undirected) graph. A \emph{$q$-coloring} of $G$ is an attribution of a \emph{color} in $\{1,2,\ldots,q\}$ to each vertex of $G$. A $q$-coloring is called \emph{proper} if any pair of adjacent vertices get different colors. 
The \emph{chromatic polynomial} of $G$ is the polynomial $\chi_G$ such that for all positive integers $q$, the evaluation $\chi_G(q)$ is the number of proper $q$-colorings.

In this article we provide a combinatorial interpretation for the evaluations of the polynomial $\chi_G(q)$ and of its derivatives $\chi_G^{(i)}(q)$ at negative integers. Let us state this result. Recall that an orientation of $G$ is called \emph{acyclic} if it does not have any directed cycle. 
A \emph{source} of an orientation is a vertex without any ingoing edge. For a set $U$ of vertices of $G$, we denote $G[U]$ the \emph{subgraph of $G$ induced by $U$}, that is, the graph having vertex set $U$ and edge set made of the edges of $G$ with both endpoints in $U$. The following is our main result about $\chi_G$, where we use the notation $[n]:=\{1,\ldots,n\}$ for a positive integer $n$, and the convention $[0]=\emptyset$.

\begin{thm}\label{thm:main}
Let $G$ be a graph with vertex set $[n]$. 
For any non-negative integers $i,j$, $(-1)^{n-i}\chi_G^{(i)}(-j)$ counts the number of tuples $((V_1,\ga_1),\ldots,(V_{i+j},\ga_{i+j}))$ such that 
\begin{compactitem}
\item $V_1,\ldots,V_{i+j}$ are disjoint subsets of vertices, such that $\bigcup_k V_k=[n]$,
\item for all $k\in [i+j]$, $\ga_k$ is an acyclic orientation of $G[V_k]$,
\item for $k\in [i]$, $V_k\neq \emptyset$ and $\ga_k$ has a unique source which is the vertex $\min(V_k)$.
\end{compactitem}
\end{thm}

We will also prove a generalization of Theorem~\ref{thm:main} (see Theorem~\ref{thm:generalized}), and a refinement at the level of the chromatic symmetric function (see Theorem~\ref{thm:main-symmetric}). As we explain in Section~\ref{sec:related}, the cases $i=0$ and $j=0$ of Theorem~\ref{thm:main} are classical results due to Stanley~\cite{stanley_acyclic} and to Greene and Zaslavsky~\cite{greenezaslavsky} respectively. However these special cases are usually presented in terms of colorings (instead of partitions of the vertex set) and global acyclic orientations (instead of suborientations). A version of Theorem~\ref{thm:main} in this spirit is given in Corollary~\ref{cor:equivThm1}. 

\fig{width=.9\linewidth}{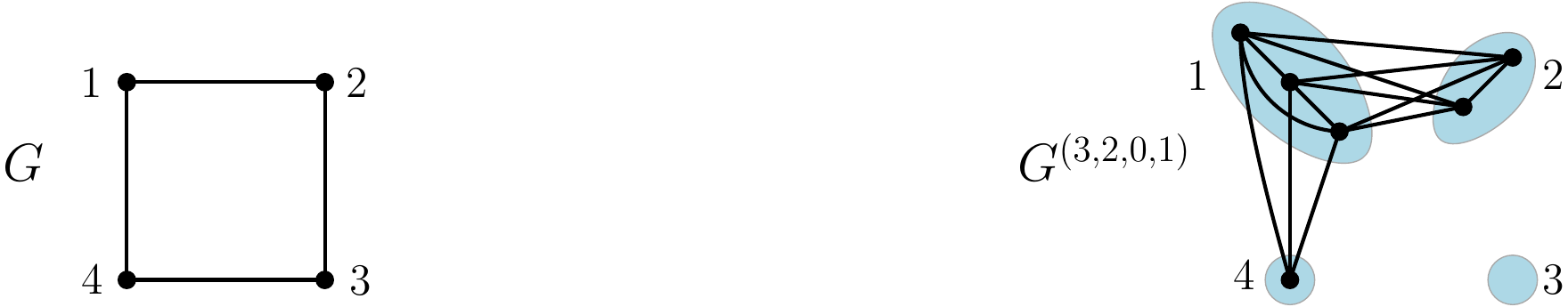}{Left: A graph $G$ on 4 vertices, having chromatic polynomial $\chi_G(q)=q^4-4q^3+6q^2-3q$. Right: The graph $G^{(3,2,0,1)}$.}

Let us illustrate Theorem \ref{thm:main} for the graph $G$ represented in Figure~\ref{fig:example-graph}. For $i=j=1$, one needs to counts the pairs $((V_1,\ga_1),(V_2,\ga_2))$, where $V_1\uplus V_2=\{1,2,3,4\}$, $\ga_1$ is an acyclic orientations of $G[V_1]$ with unique source $\min(V_1)$, and $\ga_2$ is any acyclic orientation of $G[V_2]$. The number of valid pairs with $V_1$ of size 1 (resp. 2, 3, 4) is 16 (resp. 8, 4, 3). This gives a total of 31 pairs which, as predicted by Theorem \ref{thm:main}, is equal to $-\chi_G'(-1)$.

In many ways, it feels like Theorem~\ref{thm:main} should have been discovered earlier. Our proof is based on the theory of heaps, which takes its root in the work of Cartier and Foata~\cite{cartierfoata}, and has been popularized by Viennot~\cite{viennot:heaps_newyork}. In fact, our proof is in the same spirit as the one used by Gessel in~\cite{gessel:chromatic-heaps}, and subsequently by Lass in~\cite{lass:chromatic-heaps} (see also the recent preprint \cite{Deb:chromatic-heap}). It consists in showing that well-known counting lemmas for heaps imply a relation between proper colorings and acyclic orientations. We recall the basic theory of heaps and their enumeration in Section~\ref{sec:heaps}. Theorem~\ref{thm:main} is proved in Section~\ref{sec:proof}. In Section~\ref{sec:related}, we discuss some reformulations, and extensions of Theorem~\ref{thm:main} and their relations to the results in~\cite{gessel:chromatic-heaps,greenezaslavsky,lass:chromatic-heaps,stanley_acyclic}. In Section~\ref{sec:sym}, we lift Theorem~\ref{thm:main} at the level of the chromatic symmetric function.

\section{Heaps: definition and counting lemmas} \label{sec:heaps}
In this section we recall the basic theory of heaps. We fix a graph $G=([n],E)$ throughout.

\subsection{Heaps of pieces}
\label{sub:heaps}
We first define $G$-heaps. Our (slightly unconventional) definition is in terms of acyclic orientations of a graph related to $G$. Let $\NN=\{0,1,2,\ldots\}$ be the set of non-negative integers. For a tuple $\mm=(m_1,\cdots,m_n)\in\NN^n$, we define a graph $G^\mm:=(V^\mm,E^\mm)$ with vertex set 
$$\ds V^\mm:=\{v_i^k\}_{i\in[n],~ k\in [m_i]},$$
and edge set defined as follows:
\begin{compactitem}
\item for every vertex $i\in[n]$ of $G$ there is an edge of $G^\mm$ between $v_i^k$ and $v_i^{\ell}$ for all $k,\ell\in[m_i]$,
\item for every pair of adjacent vertices $i,j\in[n]$ of $G$ there is an edge of $G^\mm$ between $v_i^k$ and $v_{j}^{\ell}$ for all $k\in[m_i]$ and all $\ell\in[m_j]$.
\end{compactitem}
The notation $G^\mm$ is illustrated in Figure \ref{fig:example-graph} (right).

\begin{definition} \label{def:heap_acyclic} 
A \emph{$G$-heap} of \emph{type} $\mm$ is an acyclic orientation of the graph $G^\mm$ such that for all $i\in [n]$ and for all $1\leq k<\ell\leq m_i$ the edge between $v_i^k$ and $v_i^{\ell}$ is oriented toward $v_i^{\ell}$. The vertices $v_i^k$ of $G^\mm$ are called \emph{pieces of type $i$} of the $G$-heap.
\end{definition}

\begin{remark}
A more traditional definition of heaps is in terms of partially ordered sets. Namely, a \emph{$G$-heap of type $\mm$} is commonly defined as a partial order $\prec$ on the set $V^\mm$ such that 
\begin{compactenum}
\item[(a)] for any vertex $i\in [n]$, $v_i^1\prec v_i^2\prec\ldots \prec v_i^{m_i}$ ,
\item[(b)] for any adjacent vertices $i,j\in [n]$, the set $\{v_i^k\}_{k\in{[m_i]}}\cup \{v_{j}^\ell\}_{\ell\in{[m_j]}}$ is totally ordered by $\prec$,
\item[(c)] and the order relation is the transitive closure of the relations of type~(a) and~(b).
\end{compactenum}
It is clear that this traditional definition is equivalent to Definition~\ref{def:heap_acyclic}: the relation $\prec$ between vertices in $V^\mm$ simply encodes the existence of a directed path between these vertices. In fact, Definition~\ref{def:heap_acyclic} already appears in~\cite[Definition (c), p.545]{viennot:heaps_newyork}. 

Heaps were originally introduced to represent elements in a partially commutative monoid~\cite{cartierfoata}. We refer the interested reader to~\cite{Krattenthaler:Heaps,viennot:heaps_newyork} for more information about heaps.
\end{remark}

Recall that for an oriented graph, a vertex without ingoing edges is called a \emph{source}, and a vertex without outgoing edges is called a \emph{sink}.
A piece of a heap $\he$ is called \emph{minimal} (resp. \emph{maximal}) if it is a source (resp. sink) in the acyclic orientation $\he$ of $G^\mm$.
A heap is called \emph{trivial} if every piece is both minimal and maximal (which occurs when $G^\mm$ consists of isolated vertices). A heap is a \emph{pyramid}\footnote{This is sometimes called \emph{upside-down pyramid}.} if it has a unique minimal piece. 

Next, we define the generating functions of heaps, trivial heaps and pyramids. 
Let $\xx=(x_1,\ldots,x_n)$ be commutative variables.
Let $\mH$, $\mT$, and $\mP$ be the set of heaps, trivial heaps, and pyramids respectively. We define
\begin{equation}\label{eq:GF}
H(\xx)=\sum_{\he\in\mH}\xx^\he,~~~~ T(\xx)=\sum_{\he\in\mT}\xx^\he,~ \textrm{ and }~~~ P(\xx)=\sum_{\he\in\mP} \frac{\xx^\he}{\she},
\end{equation}
where $\she$ is the number of pieces in the heap $\he$, and $\ds \xx^\he:=\prod_{i=1}^n x_i^{\#\textrm{ pieces of type }i\textrm{ in }\he}$.
In other words, these generating functions, which are formal power series in $x_1,\ldots,x_n$, count heaps according to the number of pieces of each type.

\begin{example} For the graph $G$ represented in Figure \ref{fig:example-graph}, the generating functions  $T$, $H$, $P$ have the following expansions:
\begin{eqnarray*}
T(\xx)&=&1+x_1+x_2+x_3+x_4+x_1x_3+x_2x_4.\\
H(\xx)&=&1+x_1+x_2+x_3+x_4\\&& +x_1^2+x_2^2+x_3^2+x_4^2+x_1x_3+x_2x_4+2x_1x_2+2x_2x_3+2x_3x_4+2x_4x_1+\cdots\\
P(\xx)&=&x_1+x_2+x_3+x_4\\&&+\frac{1}{2}\left(x_1^2+x_2^2+x_3^2+x_4^2+2x_1x_2+2x_2x_3+2x_3x_4+2x_4x_1\right)+\cdots
\end{eqnarray*}
\end{example}

\subsection{Enumeration of heaps}
\label{sub:heap_enumeration}

We now state the classical relation between $H(\xx)$, $T(\xx)$, and $P(\xx)$. For a scalar $r$, we use the notation $r\xx:=(r\,x_1,\ldots,r\,x_n)$.
\begin{thm}[\cite{viennot:heaps_newyork}]\label{thm:heap}
The generating functions of heaps, trivial heaps and pyramids are related by
\begin{equation}\label{eq:heap-trivial}
H(\xx)=\frac{1}{T(-\xx)},
\end{equation}
and
\begin{equation}\label{eq:pyramid-trivial}
P(\xx)=-\ln(T(-\xx)).
\end{equation}
\end{thm}
Equations (\ref{eq:heap-trivial}-\ref{eq:pyramid-trivial}) are identities for formal power series in $x_1,\ldots,x_n$. Observing that $T(\xx)$ has constant term 1 (corresponding to the empty heap), the right-hand side of~\eqref{eq:heap-trivial} should be understood as $\sum_{n=0}^\infty (1-T(-\xx))^n$ and the right-hand side of~\eqref{eq:pyramid-trivial} should be understood as $\sum_{n=1}^\infty (1-T(-\xx))^n/n$.

Theorem~\ref{thm:heap} will be proved using the following classical result. 
\begin{lemma}[\cite{viennot:heaps_newyork}]\label{lem:heaps} 
Let $S\subseteq [n]$. 
Let $\mH_{S}$ be the set of $G$-heaps such that every minimal piece has type in $S$, and let $\mT_{\ov S}$ be the set of trivial $G$-heaps such that every piece has type in $[n]\setminus S$. Then the generating functions 
$$H_{S}(\xx)=\sum_{\he\in \mH_S} \xx^\he,~ \textrm{ and }~ T_{\ov S}(\xx)=\sum_{\he\in \mT_{\ov S}} \xx^\he,$$
are related by 
\begin{equation}\label{eq:lem-heaps}
H_{S}(\xx)=\frac{T_{\ov S}(-\xx)}{T(-\xx)}.
\end{equation}
\end{lemma}
Let us give a sketch of the standard proofs of Lemma~\ref{lem:heaps} and Theorem~\ref{thm:heap}. Observe first that the identity~\eqref{eq:lem-heaps} is equivalent to 
\begin{equation}\label{eq:lem-heaps-proof}
\sum_{(\he_1,\he_2)\in \mH_S\times \mT}(-1)^{|\he_2|}\xx^{\he_1}\xx^{\he_2}=\sum_{\he\in \mT_{\ov S}} (-1)^{|\he|}\xx^\he.
\end{equation}
We now explain how to prove~\eqref{eq:lem-heaps-proof} using a \emph{sign-reversing involution} on $\mH_S\times \mT$.
Given $\he_1\in \mH_S$ of type $\mm_1$ and $\he_2\in \mT$ of type $\mm_2$, we define $\he:=\he_1*\he_2$ as the heap of type $\mm=\mm_1+\mm_2$ obtained from $\he_1$ by adding the pieces of $\he_2$ as new sinks. More precisely, $\he$ is the orientation of $G^\mm$ such that the restriction to $G^{\mm_1}$ is $\he_1$ and the vertices in $V^\mm\setminus V^{\mm_1}$ are sinks. 
Now, we fix a heap $\he$, and look at the set $\mS_\he$ of pairs $(\he_1,\he_2)\in \mH_S\times \mT$ such that $\he_1*\he_2=\he$. If $\he\not\in \mT_{\ov S}$, one can define a simple sign reversing involution on $\mS_\he$ in order to prove that the contributions of the pairs $(\he_1,\he_2)\in \mS_\he$ to~\eqref{eq:lem-heaps-proof} cancel out. This involution simply transfers a canonically-chosen piece of $\he$ between $\he_1$ and $\he_2$ (one can transfer any maximal piece of $\he$ which either has type in $S$ or is not minimal, so a canonical choice is to transfer the piece of minimal type among those).  If $\he\in \mT_{\ov S}$, then $\mS_\he=\{(\epsilon,\he)\}$, where $\epsilon$ is the empty heap, hence the contribution of $\mS_\he$ to~\eqref{eq:lem-heaps-proof} is 1. This proves Lemma~\ref{lem:heaps}.

To prove Theorem~\ref{thm:heap}, observe first that~\eqref{eq:heap-trivial} is the special case $S=[n]$ of~\eqref{eq:lem-heaps}. It remains to prove~\eqref{eq:pyramid-trivial}. Let $t$ be an indeterminate. By differentiating the series $P(t\,\xx)$ (formally) with respect to $t$ we get 
$$t\cdot\frac{\partial}{\partial t}P(t\,\xx)=t\cdot\frac{\partial}{\partial t}\sum_{\he\in \mP}\frac{t^\she}{\she}\xx^\he=\sum_{\he\in \mP}t^\she\xx^\he.$$
We now use the partition $\mP=\biguplus_{k\in[N]}(\mH_{\{k\}}\setminus \{\eps\})$, where $\eps$ is the empty heap. This, together with~\eqref{eq:lem-heaps} gives
\begin{equation}\nonumber
t\cdot \frac{\partial}{\partial t}P(t\,\xx)=\sum_{k=1}^n\left(H_{\{k\}}(t\,\xx)-1\right)=\sum_{k=1}^n\frac{T_{\ov{\{k\}}}(-t\,\xx)-T(-t\,\xx)}{T(-t\,\xx)}=\frac{-1}{T(-t \xx)}\sum_{k=1}^n T_{k}(-t\,\xx),
\end{equation}
where $$T_k(\xx)=\sum_{\substack{\he\in \mT\\ \textrm{containing a piece of type }k}}\xx^\he.$$
Finally, we observe that $\ds \sum_{k=1}^n T_{k}(\xx)=\sum_{\he\in\mT}\she \xx^\he$. This gives
$$\frac{\partial}{\partial t}P(t\,\xx)=\frac{1}{T(-t \xx)}\cdot \sum_{\he\in\mT} \she (-t)^{\she-1}\xx^\he=\frac{-1}{T(-t \xx)}\cdot \frac{\partial}{\partial t}T(-t\,\xx),$$
which, upon integrating (formally) with respect to $t$, gives~\eqref{eq:pyramid-trivial}.


\section{Heaps, colorings, and orientations: proof of Theorem~\ref{thm:main}}\label{sec:proof}
This section is dedicated to the proof of Theorem~\ref{thm:main}. We fix a graph $G=([n],E)$ throughout.
\begin{notation}
We denote by $R[[\xx]]$ the ring of power series in $x_1,\ldots,x_n$ with coefficients in a ring $R$. For a tuple $\mm=(m_1,\ldots,m_n)\in\NN^n$, we denote $\xx^\mm=x_1^{m_1}\cdots x_n^{m_n}$. For a power series $F(\xx)\in R[[\xx]]$, we denote by $[\xx^\mm]F(\xx)$ the coefficient of $\xx^\mm$ in $F(\xx)$.
\end{notation}

The first step is to express the chromatic polynomial of $G$ in terms of trivial heaps. 
\begin{lemma}\label{lem:chrom1}
Let $T(\xx)$ be the generating function of trivial $G$-heaps defined in~\eqref{eq:GF}, and let $q$ be an indeterminate. Then,
\begin{equation}\label{eq:chrom}
\chi_G(q)=[x_1\cdots x_n]T(\xx)^q.
\end{equation}

\end{lemma}
The right-hand side in~\eqref{eq:chrom} has to be understood as the coefficient of $x_1\cdots x_n$ in the series $\ds \exp(q\ln(T(\xx)):=\sum_{k=0}^\infty \frac{(q\ln(T(\xx)))^k}{k!} \in \QQ[q][[\xx]]$.
\begin{proof}
Recall that a set of vertices $V\subseteq [n]$ is called \emph{independent} if the vertices in $V$ are pairwise non-adjacent. There is an obvious equivalence between independent sets and trivial heaps, hence $T(\xx)$ can be thought as the generating function of independent sets.

Let $q$ be a positive integer. Observe that for any proper $q$-coloring, the set of vertices of color $i\in[q]$ is an independent set. In fact, upon denoting $V_i$ the set of vertices of color~$i$, it is clear that a proper $q$-coloring can equivalently be seen as a $q$-tuple $(V_1,\ldots,V_q)$ of independent sets of vertices, which are disjoint and such that $\bigcup_{k\in [q]} V_k=[n]$. This immediately implies that~\eqref{eq:chrom} holds for the positive integer $q$. Since both sides of~\eqref{eq:chrom} are polynomials in $q$, the identity holds for an indeterminate $q$.
\end{proof}

Upon differentiating~\eqref{eq:chrom} $i$ times one gets 
$$\chi_G^{(i)}(q)=\frac{\partial^i}{\partial q^i}[x_1\cdots x_n]T(\xx)^q=[x_1\cdots x_n]\frac{\partial^i}{\partial q^i}\exp\left(q\ln(T(\xx))\right)=[x_1\cdots x_n]\ln(T(\xx))^i T(\xx)^q.$$ 
In the right-hand side of the above equation, we are are extracting a coefficient of degree $n$, hence this expression is invariant under changing $\xx$ into $-\xx$ and multiplying by $(-1)^n$. Hence
$(-1)^n\chi_G^{(i)}(q)=[x_1\cdots x_n]\ln(T(-\xx))^i\, T(-\xx)^q$, and for a non-negative integer $j$, 
\begin{eqnarray}
(-1)^{n-i}\chi_G^{(i)}(-j)&=&[x_1\cdots x_n](-\ln(T(-\xx)))^i(T(-\xx))^{-j}\nonumber\\
&=& [x_1\cdots x_n]P(\xx)^i H(\xx)^j,\label{eq2}
\end{eqnarray}
where the last equality follows from Theorem~\ref{thm:heap}.

The next step is to relate heaps and pyramids to acyclic orientations. For a set $V\subseteq [n]$, let $\xx^V$ be the monomial $x_1^{\de_1}\cdots x_n^{\de_n}$, where $\de_i=1$ if $i\in V$ and $\de_i=0$ otherwise. 

\begin{lemma}\label{lem:chrom2} Let $V\subseteq [n]$.
The generating function $H(\xx)$ and $P(\xx)$ defined in~\eqref{eq:GF} satisfy
\begin{align}
[\xx^V]H(\xx)&=\# \textrm{ acyclic orientations of }G[V], \label{eq:acyc} \\
[\xx^V]P(\xx)&=\# \textrm{ acyclic orientations of }G[V] \textrm{ with unique source }\min(V),\label{eq:acyc-rooted} 
\end{align}
where the right-hand side of~\eqref{eq:acyc-rooted} is interpreted as 0 if $V=\emptyset$.
\end{lemma}

\ob{
\begin{proof}
Let $\mm=(\de_1,\ldots, \de_n)$, where $\de_i=1$ if $i\in V$, and $\de_i=0$ otherwise. Observe that $G[V]$ is isomorphic to the graph $G^\mm$.
By definition of $H$, the coefficient $[\xx^V]H(\xx)$ counts the $G$-heaps of type $\mm$, or equivalently the acyclic orientations of $G[V]$. This proves~\eqref{eq:acyc}.  Let us now assume $V\neq\emptyset$. 
By definition of $P$, one gets $[\xx^V]P(\xx)=\frac{B}{|V|}$, where $B$ is the number pyramids of type $\mm$, or equivalently the number of acyclic orientations of $G[V]$ with a single source.
For $i\in V$, let $\mB_i$ be the set of acyclic orientations of $G[V]$ with unique source $i$. It is not hard to see that $|\mB_i|=|\mB_j|$ for all $i,j\in V$. Indeed, a bijection between $\mB_i$ and $\mB_j$ can be constructed as follows: given $\ga\in\mB_i$, reverse all the edges of $\ga$ on any directed path from $i$ to $j$. This proves~\eqref{eq:acyc-rooted}.
\end{proof}
}


We now complete the proof of Theorem~\ref{thm:main}. 
For any non-negative integers $i,j$,~\eqref{eq2} gives 
\begin{eqnarray}
(-1)^{n-i}\chi_G^{(i)}(-j)&=&\sum_{V_1\uplus\cdots \uplus V_{i+j}=[n]}~\prod_{k=1}^{i}[\xx^{V_k}]P(\xx)\prod_{\ell=1}^j[\xx^{V_{i+\ell}}]H(\xx),\label{eq3}
\end{eqnarray}
where the sum is over the tuples of disjoint sets $V_1,\ldots,V_{i+j}$ whose union is $[n]$.
Finally, by Lemma~\ref{lem:chrom2}, the right-hand side of~\eqref{eq3} can be interpreted as in Theorem~\ref{thm:main}.

\begin{remark}
\label{rem:multicolorings}
Equation~\eqref{eq:chrom} raises the question of interpreting the other coefficients of $T(\xx)^q$ combinatorially. 
So for $\mm\in\NN^n$, let us introduce the following polynomial
\begin{equation}
\label{eq:chi_m}
\chi_{G,\mm}(q):=[\xx^\mm]T(\xx)^q,
\end{equation}
so that $\chi_G(q)=\chi_{G,\1^n}(q)$. 
It is easy to interpret~\eqref{eq:chi_m} combinatorially: for any positive integer $q$, $\chi_{G,\mm}(q)$ counts the functions $f$ from the vertex set $[n]$ to the power set $2^{[q]}$ such that for any vertex $i\in [n]$, $|f(i)|=m_i$ and for adjacent vertices $i,j\in [n]$ of $G$, the sets $f(i)$ and $f(j)$ are disjoint. These are known as \emph{proper multicolorings} of $G$ of type $\mm$~\cite{ gasharov,stanley_graphcoloring}. 

Now, recalling the definition of the graph $G^\mm$, it is easy to see that 
\[\chi_{G,\mm}(q)=\frac{\chi_{G^\mm}(q)}{\mm!},\]
where $\mm!:=m_1!\cdots m_n!$. Indeed, there is a clear $\mm!$-to-1 correspondence between the proper colorings of $G^\mm$ and the multicolorings of $G$ of type $\mm$: to a proper coloring of $G^\mm$ one associates the multicoloring $f$ of $G$, where $f(i)$ is the set of colors used on the vertices $\{v_i^k\}_{k\in [m_i]}$ of $G^\mm$. On the one hand, this shows that all the coefficients of $T(\xx)^q$ are chromatic polynomials, up to a multiplicative constant. On the other hand, using~\eqref{eq:chi_m} and Theorem~\ref{thm:heap}, we get $(-1)^{|\mm|}\chi_{G,\mm}(-1)=[\xx^\mm]H(x)$ which is the number of heaps of type~$\mm$. Hence general heaps come up naturally in the context of proper multicolorings.
\end{remark}

\begin{remark} 
Various generalizations of the chromatic polynomials have been considered in the literature, and the above technique can be used to give a reciprocity theorem for those. In particular, the \emph{bivariate chromatic polynomial} $\chi_G(q,r)$ is defined in~\cite{bivariate} as the polynomial whose evaluation at $(q,r)\in\NN^2$ counts the $(q+r)$-colorings of $G$ such that adjacent vertices cannot receive the same color in $[q]$. 
It is easy to express this polynomial in terms of heaps, and use similar techniques as above to obtain a combinatorial interpretation for $(-1)^n\chi_G(-j,-k)$.
\ob{Namely, this counts the number of tuples $((V_1,\ga_1),\ldots, (V_j,\ga_j),V_{j+1},\ldots,V_{j+k})$ such that $\biguplus_{i=1}^{j+k}V_i=[n]$ and for all $i\in[j]$, $\ga_i$ is an acyclic orientation of $G[V_i]$. One can similarly get an interpretation for the evaluations $\frac{\partial^i}{\partial q^i}\chi_G(-j,-k)$ of the derivatives with respect to $q$.}
\end{remark}

\section{Special cases, and extensions}\label{sec:related}
In this section we discuss some reformulations and extensions of Theorem~\ref{thm:main}.

\subsection{Specializations of Theorem~\ref{thm:main}, and reformulation.} \label{sec:specializations}
We first establish the relation between Theorem~\ref{thm:main} and the results from ~\cite{greenezaslavsky,stanley_acyclic}.

Let us recall the seminal result of Stanley~\cite{stanley_acyclic} about the negative evaluations of the chromatic polynomial. Let $G=(V,E)$ be a graph, and let $\ga$ be an orientation of $G$. We say that a $q$-coloring of $G$ (that is, a function $f:V\to [q]$) has no \emph{$\ga$-descent} if the colors (that is, the values of $f$) never decrease strictly along the arcs of $\ga$.

\begin{prop}[{\cite[Theorem 1.2]{stanley_acyclic}}]\label{prop:stanley}
Let $G$ be a graph with $n$ vertices, and let $j$ be a non-negative integer. Then, $(-1)^n\chi_G(-j)$ is the number of pairs $(\ga,f)$, where $\ga$ is an acyclic orientation of $G$, and $f$ is a $j$-coloring without $\ga$-descent. In particular, $(-1)^n\chi_G(-1)$ is the number of acyclic orientations of $G$.
\end{prop}
As we now explain, Proposition~\ref{prop:stanley} is equivalent to the case $i=0$ of Theorem~\ref{thm:main}. Let $\mC_j$ be the set of pairs $(\ga,f)$, where $\ga$ is an acyclic orientation of $G$, and $f$ is a $j$-coloring without $\ga$-descent. A $j$-coloring $f$ can be encoded by the tuple $(V_1,\ldots,V_j)$, where $V_k=f^{-1}(k)$ is the set of vertices of color $k$. Now given $f$, the orientations $\ga$ such that $(\ga,f)\in \mC_j$ are such that for all $k\in[j]$ the restriction $\ga_k$ of $\ga$ to $G[V_k]$ is acyclic, and for all $\ell>k$ every edge between $V_k$ and $V_\ell$ is oriented toward its endpoint in $V_\ell$. These two conditions are easily seen to be sufficient. Hence, pairs $(\ga,f)\in \mC_j$ are uniquely determined by choosing the ordered partition $(V_1,\ldots,V_j)$ and the acyclic orientations $\ga_1,\ldots,\ga_j$ of $G[V_1],\ldots,G[V_j]$. This shows the equivalence between Proposition~\ref{prop:stanley} and the case $i=0$ of Theorem~\ref{thm:main}.\\

Next we recall the result of Greene and Zaslavsky~\cite{greenezaslavsky} about the coefficients of the chromatic polynomial. We need to define the \emph{source-components} of an acyclic orientation $\ga$ of $G=([n],E)$. 
For $i\in [n]$, let $R_i$ be the set of vertices reachable from $i$ by a directed path of $\ga$ (with $i\in R_i$). We now define some subsets of vertices $S_1,S_2,\ldots$ recursively as follows. For $k\geq 1$, if $\bigcup_{i<k}S_i=[n]$, then we define $S_k=\emptyset$. Otherwise, we define $S_k=R_m\setminus \bigcup_{i<k}S_i$, where $m=\min\left([n]\setminus \bigcup_{i<k}S_i\right)$. The non-empty subsets $S_k$ are called the \emph{source-components} of $\ga$. The source components are represented for various acyclic orientations in Figure \ref{fig:example-graph}.
Note that the source-components of an orientation $\ga$ form an ordered partition of $[n]$, and that the restriction of $\ga$ to each subgraph $G[S_k]$ is an acyclic orientation with single source $\min(S_k)$.

\fig{width=\linewidth}{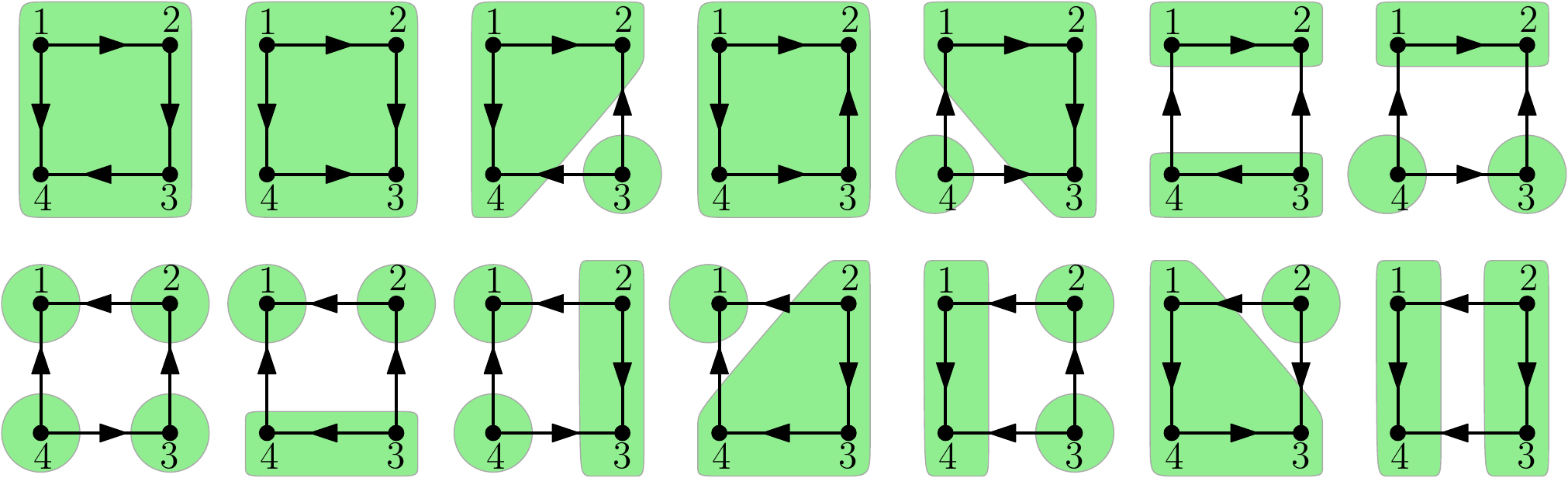}{The source-components of the 14 acyclic orientations of the graph of Figure~\ref{fig:example-graph}.}


\begin{prop}[{\cite[Theorem 7.4]{greenezaslavsky}}]\label{prop:greenezaslavsky}
Let $G=([n],E)$ be a graph, and let $i$ be a non-negative integer. Then, $(-1)^{n-i}[q^i]\chi_G(q)$ is the number of acyclic orientations of $G$ with exactly $i$ source-components. In particular, $(-1)^{n-1}[q^1]\chi_G(q)$ is the number of acyclic orientations with single source 1.
\end{prop}

\begin{example} The graph $G$ in Figure~\ref{fig:example-graph}, has 1 (resp. 4, 6, 3) acyclic orientations with 4 (resp. 3, 2, 1) source-components. This matches the coefficients of $\chi_G(q)=q^4-4q3 + 6q2- 3q$.
\end{example}

As we now explain, Proposition~\ref{prop:greenezaslavsky} is equivalent to the case $j=0$ of Theorem~\ref{thm:main}. Let $\mA_i$ be the set of acyclic orientations of $G$ with exactly $i$ source-components. Let $\ga\in \mA_i$, and let $S_1,\ldots,S_i$ be its source-components. The sets $S_1,\ldots,S_i$ clearly satisfy
\begin{compactitem}
\item[(i)] $S_1,\ldots,S_i$ are disjoint sets and $\bigcup_{k=1}^i S_k=[n]$,
\item[(ii)] for all $k\in[i]$ the restriction $\ga_k$ of $\ga$ to the subgraph $G[S_k]$ is an acyclic orientation with single source $\min(S_k)$, 
\item[(iii)] for all $\ell>k$, any edge between $S_k$ and $S_\ell$ is directed toward its endpoint in~$S_k$.
\item[(iv)] $\min(S_1)<\min(S_2)<\cdots<\min(S_i)$,
\end{compactitem}
These conditions are easily seen to be sufficient: an acyclic orientation $\ga$ has source-components $S_1,\ldots,S_i$ if and only if the conditions (i-iv) hold. Moreover, the tuple $((S_1,\ga_1),\ldots, (S_i,\ga_i))$ uniquely determines $\ga\in \mA_i$. Hence, Proposition~\ref{prop:greenezaslavsky} can be interpreted as stating that $(-1)^{n-i}[q^i]\chi_G(q)$ is the number of tuples $((S_1,\ga_1),\ldots, (S_i,\ga_i))$ satisfying (i-iv). Upon permuting the indices $\{1,\ldots,i\}$, we get that $i!(-1)^{n-i}[q^i]\chi_G(q)$ is the number of tuples $((S_1,\ga_1),\ldots, (S_i,\ga_i))$ satisfying conditions (i-iii), which is exactly the case $j=0$ of Theorem~\ref{thm:main}.\\

It is not hard to combine the above discussions to show that Theorem~\ref{thm:main} is equivalent to the following statement. 

\begin{cor}\label{cor:equivThm1}
Let $G$ be a graph, let $q$ be an indeterminate, and let $i,j$ be non-negative integers. 
Then $(-1)^{n-i}[q^i]\chi_G(q-j)$ is the number of pairs $(\ga,f)$, where $\ga$ is an acyclic orientation of $G$, and $f$ is a $(j+1)$-coloring of $G$ without $\ga$-descent, such that the restriction $\ga_1$ of $\ga$ to the subgraph $G[f^{-1}(1)]$ has exactly $i$ source-components (with the special case $i=0$ corresponding to $f^{-1}(1)=\emptyset$).
\end{cor}

\subsection{Generalization of Theorem~\ref{thm:main} and relation to results by Gessel and Lass.}
In this subsection we establish a generalization of Theorem~\ref{thm:main}, which extends results from Gessel~\cite{gessel:chromatic-heaps} and Lass~\cite{lass:chromatic-heaps}.

\begin{thm}\label{thm:generalized}
Let $G=([n],E)$ be a graph. Let $d$ be a non-negative integer such that the vertices $1,2,\ldots,d$ are pairwise adjacent. 
Let $q$ be an indeterminate, and let
\begin{equation}\label{def:Pd}
\Pd(q):=\frac{\chi_G(q)}{q(q-1)\cdots(q-d+1)},
\end{equation}
with the special case $d=0$ being interpreted as $\Pzero(q)=\chi_G(q)$. Then $\Pd(q)$ is a polynomial in $q$ such that for all non-negative integers $i,j$, the evaluation $(-1)^{n-d-i}\,\Pd^{(i)}(-j)$ is the number of tuples $((V_1,\ga_1),\ldots,(V_{d+i+j},\ga_{d+i+j}))$ such that 
\begin{compactitem}
\item $V_1,\ldots,V_{d+i+j}$ are disjoint subsets of vertices, such that $\bigcup_k V_k=[n]$, and for all $k\in [d]$, $k\in V_k$,
\item for all $k\in [d+i+j]$, $\ga_k$ is an acyclic orientation of $G[V_k]$, and if $k\leq d+i$ then $V_k\neq \emptyset$ and $\ga_k$ has a unique source which is the vertex $\min(V_k)$.
\end{compactitem}
\end{thm}

Observe that the case $d=0$ of Theorem~\ref{thm:generalized} is Theorem~\ref{thm:main}. The special case $i=0$ for $d\in\{1,2\}$ was obtained by Gessel in~\cite[Thm 3.3 and 3.4]{gessel:chromatic-heaps}.

\begin{example}
For the graph $G$ represented in Figure~\ref{fig:example-graph-compo}, we have $\Ptwo(q)=q^2-3q+3$. Theorem \ref{thm:generalized} in the case $d=2,i=1,j=0$ (correctly) predicts that there are exactly  $-\Ptwo'(0)=3$ triples $((V_1,\ga_1), (V_2,\ga_2),(V_3,\ga_3))$, such that $1\in V_1$, $2\in V_2$, $V_1\uplus V_2\uplus V_3=[4]$, and for all $k\in [3]$, $\gamma_k$ is an acyclic orientation of $G[V_k]$ with unique source $\min(V_k)$. 
\end{example}

\begin{proof}
Since the vertices $1,2,\ldots,d$ are pairwise adjacent, we know that $\chi_G(k)=0$ for all $k\in\{0,\ldots,d-1\}$. Since these integers are roots of $\chi_G(q)$, this polynomial is divisible by $q(q-1)\cdots(q-d+1)$. Hence $\Pd(q)$ is a polynomial. 
We now prove the interpretation of $(-1)^{n-d-i}\Pd^{(i)}(-j)$. Fix an integer $q>d$. Note that in any proper $q$-coloring of~$G$, the vertices $1,\ldots,d$ have distinct colors in $[q]$. So it is easy to see that $\Pd(q)$ can be interpreted as the number of proper $q$-colorings such that for all $k$ in $[d]$ the vertex $k$ has color $k$. In other words, for all $k\in[d]$, the $q$-colorings counted by $\Pd(q)$ are such that the set of vertices colored $k$ are independent sets  containing the vertex $k$. 
Thus, reasoning as in the proof of~\eqref{eq:chrom}, we get the following expression of $\Pd(q)$ in terms of trivial heaps:
\begin{equation}\label{eq:chrom-gle}
\Pd(q)=[x_1\cdots x_n]\left(\prod_{k=1}^d\frac{T_k(\xx)}{T(\xx)}\right)T(\xx)^q,
\end{equation}
where $T_k(\xx)=\sum_{\he\in \mT_k}\xx^\he$ and $\mT_k$ is the set of trivial heaps containing a piece of type $k$. Again, this equation holds for an indeterminate $q$, because both sides are polynomials in $q$. Differentiating~\eqref{eq:chrom-gle} with respect to $q$ ($i$ times), and setting $q=-j$ gives 
$$(-1)^{n-d-i}\,\Pd^{(i)}(-j)=[x_1\cdots x_n]\left(\prod_{k=1}^d-\frac{T_k(-\xx)}{T(-\xx)}\right)\left(-\ln(T(-\xx))\right)^i\left(\frac{1}{T(-\xx)}\right)^{j}.$$
By Lemma~\ref{lem:heaps}, 
$\ds -\frac{T_k(-\xx)}{T(-\xx)}=\frac{T_{\ov{\{k\}}}(-\xx)}{T(-\xx)}-1=H_{\{k\}}(\xx)-1,$ 
which together with Theorem~\ref{thm:heap} gives
\begin{equation}\label{eq:Pd-negative}
(-1)^{n-d-i}\Pd^{(i)}(-j)=[x_1\cdots x_n]\left(\prod_{k=1}^dH_{\{k\}}(\xx)-1\right)\cdot P(\xx)^i \cdot H(\xx)^{j}.
\end{equation}
Observe that for any sets of vertices $S,V\subseteq [n]$,  the coefficient $[\xx^V]H_{S}(\xx)$ is the number of acyclic orientations of $G[V]$ whose sources are all in $S$. Hence, for any set $V\subseteq [n]$, 
$$
[\xx^V]\left(H_{\{k\}}(\xx)-1\right)=
\left|
\begin{array}{ll}
\#\textrm{ acyclic orientations of } G[V]\textrm{ with unique source }k, \textrm{ if }k\in V,\\[1mm] 
0 \textrm{ otherwise}.
\end{array}\right.
$$
Using this together with Lemma~\ref{lem:chrom2}, we see that~\eqref{eq:Pd-negative} gives the claimed interpretation of $(-1)^{n-d-i}\,\Pd^{(i)}(-j)$.
\end{proof}

Theorem~\ref{thm:generalized} could equivalently be stated as giving an interpretation for the coefficients of the polynomial $\Pd(q-j)$ for all $j,d\geq 0$. We will next give an interpretation for the coefficients of $\Pd(q+1)$ for all $d>0$.

Let us first recall a classical result of Crapo~\cite{Crapo:bipolar-orient}. Let $u,v$ be two adjacent vertices of a graph $G$. An acyclic orientation of $G$ is called \emph{$(u,v)$-bipolar} if it has unique source $u$ and unique sink $v$. A classical result of Crapo~\cite{Crapo:bipolar-orient} is that  $(-1)^{n}[q^1]\Pone(q+1)$ is the number of $(u,v)$-bipolar orientations of $G$ (which is independent of $u,v$). We mention that, for a connected graph, $\Pone(q+1)$ is related to the Tutte polynomial $T_G(x,y)$ by $\Pone(q+1)=(-1)^{n-1}T_G(-q,0)$, hence $(-1)^{n}[q^1]\Pone(q+1)=[x^1y^0]T_G(x,y)$. Crapo's result was recovered using the theory of heaps in~\cite[Thm 3.1]{gessel:chromatic-heaps}. In Lass~\cite[Thm 5.2]{lass:chromatic-heaps}, an interpretation was given for every coefficient of the polynomial $\Pone(q+1)$ for a \emph{connected} graph $G$. Following this lead, we obtain the following result for connected graphs having a set of $d$ pairwise adjacent vertices.

\begin{thm}\label{thm:extension-Lass}
Let $G=([n],E)$ be a connected graph. Let $d$ be a positive integer such that the vertices $1,2,\ldots,d$ are pairwise adjacent. Let $q$ be an indeterminate, and let $\Pd(q)$ be the polynomial defined by~\eqref{def:Pd}. 
Upon relabeling the vertices of $G$, one can assume that for all $k>1$ the vertex labeled $k$ is adjacent to a vertex of label less than~$k$. 
Then for all $i\geq 0$, $(-1)^{n-d-i}[q^i]\Pd(q+1)$
is the number of acyclic orientations of $G$ having exactly $d+i$ source-components such that the vertices $1,2,\ldots,d$ are in different source-components and $1$ is the unique sink. 
\end{thm}

Note that for the orientations described in Theorem~\ref{thm:extension-Lass}, the vertex 1 is necessarily alone in its source-component. In particular, in the special case $d=1$, and $i=1$ the orientations described have unique sink 1 and unique source 2, which gives Crapo's interpretation of $(-1)^{n}[q^1]\Pone(q+1)$ as counting $(2,1)$-bipolar orientations. The case $d=1$ of Theorem~\ref{thm:extension-Lass} is exactly~\cite[Thm 5.2]{lass:chromatic-heaps}. 
The case $d=2$ is equivalent to the case $d=1$ (because $\Pone(q+1)=q\,\Ptwo(q+1)$ and the vertices 1, 2 are necessarily in different source-components). The cases $d\geq 3$ are new.

\begin{proof}
Let $R_d(q)=\Pd(q+1)$, and let $c_i=[q^i]\Pd(q+1)=[q^i]R_d(q)=\frac{R^{(i)}(0)}{i!}$. By~\eqref{eq:chrom-gle}, 
\begin{equation*}
R_d(q)=[x_1\cdots x_n]\,T_1(\xx)\left(\prod_{k=2}^d\frac{T_k(\xx)}{T(\xx)}\right)T(\xx)^q,
\end{equation*}
for an indeterminate $q$. After differentiating with respect to $q$ ($i$ times) one gets
\begin{equation*}
(-1)^{n-d-i}c_{i}=[x_1\cdots x_n]\,\left(-T_1(-\xx)\right) \left(\prod_{k=2}^d-\frac{T_k(-\xx)}{T(-\xx)}\right)\frac{\left(-\ln(T(-\xx))\right)^i}{i!}.
\end{equation*}
Reasoning as in the proof of Theorem~\ref{thm:generalized} this gives:
\begin{eqnarray*}
(-1)^{n-d-i}c_{i}&=&[x_1\cdots x_n]\left(-T_1(-\xx)\right)\left(\prod_{k=2}^d(H_{\{k\}}(\xx)-1)\right)\, \frac{P(\xx)^i}{i!},\\
&=&\sum_{U\subseteq [n]}[\xx^U]\left(-T_1(-\xx)\right)\cdot [\xx^{\ov U}]\left(\prod_{k=2}^d(H_{\{k\}}(\xx)-1)\right) \frac{P(\xx)^i}{i!},
\end{eqnarray*}
where $\ov U:=[n]\setminus U$. 

For $V\subseteq [n]\setminus \{1\}$, let $\mS_V\equiv \mS_V(d,i)$ be the set of acyclic orientations of $G[V]$ having $d+i-1$ source-components, such that $2,\ldots,d$ are in different source-components (with $\mS_V=\emptyset$ whenever $V$ does not contain $\{2,\ldots,d\}$).
Reasoning as before, we see that  $|\mS_V|=\ds [\xx^V]\left(\prod_{k=2}^d(H_{\{k\}}(\xx)-1)\right)\frac{P(\xx)^i}{i!}$.
Hence, using the fact that $T_1$ is the generating function of the set $I$ of independent sets of $G$ containing the vertex 1, we get
\begin{equation}\label{eq:signed}
(-1)^{n-d-i}c_{i}=\sum_{U\in I}(-1)^{|U|-1}|\mS_{\ov U}|.
\end{equation}

We will now simplify this expression by defining a \emph{sign-reversing involution} $\phi$ on the set $\mS:=\{(U,\ga)~|~U\in I,~\ga\in \mS_{\ov U}\}$. Given $(U,\ga)\in \mS$ consider the orientation $\ov \ga$ which is the extension of~$\ga$ to the full graph $G$ obtained by orienting every edge incident to a vertex $u\in U$ toward $u$. It is not hard to see that $\ov \ga$ has $d+i$ source-components $S_1,\ldots,S_{d+i}$, such that $S_1=\{1\}$ and $S_2\setminus U,\ldots,S_{d+i}\setminus U$ are the source-components of~$\ga$. Indeed, it is clear that the first source-component $S_1$ is $\{1\}$ because 1 is a sink, and moreover no vertex $u\in U\setminus \{1\}$ can be the source of a source-component because $u$ is adjacent to a vertex with smaller label. 

We now define $\phi$ on $\mS$. Let $(U,\ga)\in \mS$, and let $Z$ be the set of sinks of $\ov \ga$. Note that $U\subseteq Z$ and $Z\in I$. If $Z=\{1\}$, then define $\phi(U,\ga)=(U,\ga)$. Otherwise we set $s=\min(Z \setminus \{1\})$ and consider two cases. If $s\in U$, we define $\phi(U,\ga)=(U\setminus \{s\},\ga')$, where $\ga'$ is the extension of $\ga$ to $G[\ov U \cup\{s\}]$ obtained by orienting every edge incident to $s$ toward $s$. If $s\notin U$, we define $\phi(U,\ga)=(U\cup \{s\},\ga')$, where $\ga'$ is the restriction of $\ga$ to $G[\ov U \setminus\{s\}]$. 

We know from the above discussion that in every case $\phi(U,\ga)\in \mS$. Moreover it is clear that $\phi$ is an involution (because the orientation $\ov \ga$ is unchanged by $\phi$), and that if $Z\neq \{1\}$, the contribution of the pairs $(U,\ga)$ and $\phi(U,\ga)$ to the right-hand side of~\eqref{eq:signed} will cancel out. Hence, the right-hand side of~\eqref{eq:signed} is the cardinality of the set $\mS'$ of pairs $(U,\ga)\in\mS$ such that $Z=U=\{1\}$. This gives the claimed interpretation of $(-1)^{n-d-i}c_{i}$ (upon identifying each element $(\{1\},\ga)$ in $\mS'$ with the orientation $\ov \ga$ of $G$ which is the extension of $\ga$ to $G$  obtained by orienting every edge incident to $1$ toward $1$).
\end{proof}

\section{Chromatic symmetric function}\label{sec:sym}
In this section we consider the chromatic symmetric function defined by Stanley in~\cite{stanley_symmetric}, and we obtain a symmetric function refinement of Theorem~\ref{thm:main}, as well as a ``superfication'' extension.


Let $G=([n],E)$ be a graph. We consider colorings of $G$ with colors in the set $\PP:=\{1,2,3,\ldots\}$ of positive integers. A function $f:V\to\PP$ is called  $\PP$\emph{-coloring}, and as before $f$ is said to be \emph{proper} if adjacent vertices get different colors. Let $\zz=(z_1,z_2,\ldots)$ be a set of variables indexed by $\PP$. The \emph{chromatic symmetric function} of~$G$ is the generating function of its proper $\PP$-colorings counted according to the number of times each color is used:
$$X_G(\zz)=\sum_{f \textrm{ proper } \PP\textrm{-coloring}}~~~\prod_{v\in[n]}z_{f(v)}.$$
Observe that $X_G(\zz)$ is a homogeneous symmetric function of degree $n$ in $\zz$, and that for every positive integer $j$,
\begin{equation}\label{eq:X-chi}
X_G(\1^j)=\chi_G(j),
\end{equation}
where $\1^j$ is the evaluation obtained by setting $z_i=1$ for all $i\in [j]$, and $z_i=0$ for all $i> j$.

\begin{example} \label{exp:chrom-sym}
For the graph $G$ represented in Figure \ref{fig:example-graph}, the chromatic symmetric function is easily seen to be
\begin{align*}
X_G(\zz)&=24\sum_{1\leq i<j<k<l}z_iz_jz_kz_l+4\sum_{1\leq i<j<k}(z_i^2z_jz_k+ z_iz_j^2z_k+ z_iz_jz^2_k) +2 \sum_{1\leq i<j} z_i^2z_j^2
\end{align*}
In particular, one gets
\[X_G(\1^j)=24\binom{j}{4}+12\binom{j}{3}+2\binom{j}{2}=j^4-4j^3+6j^2-3j,\]
which indeed coincides with the expression of $\chi_G(j)$ given in the caption of Figure~\ref{fig:example-graph}.
\end{example}

In~\cite{stanley_symmetric,stanley_graphcoloring} Stanley establishes many beautiful properties of $X_G$. Our goal is to recover and extend some of these results using the machinery of heaps. The starting point is the symmetric function analogue of Lemma~\ref{lem:chrom1}:
\begin{equation}\label{eq:sym-heap}
X_G(\zz)=[x_1\cdots x_n]\prod_{i= 1}^\infty T(z_i\xx),
\end{equation}
where $T(\xx)$ is the generating function of trivial $G$-heaps.

We first discuss the result of applying the \emph{duality mapping} to $X_G$. We recall some basic definitions. For a field $K$ of characteristic $0$, we denote by $\Sym_K(\zz)$ the algebra of symmetric functions in $\zz$, with coefficients in $K$. Hence, $X_G(\zz)\in \Sym_\QQ(\zz)\subseteq \Sym_K(\zz)$.
Let $e_k,h_k,p_k$ be the \emph{elementary}, \emph{complete} and \emph{power-sum} symmetric functions, which are defined by $e_0=h_0=p_0=1$, and for $k\in\PP$, 
$$e_k=\sum_{i_1<\cdots<i_k\in \PP}z_{i_1}\cdots z_{i_k},~\quad h_k=\sum_{i_1\leq \cdots\leq i_k\in \PP}z_{i_1}\cdots z_{i_k},~~\textrm{and}~\quad p_k=\sum_{i\in \PP}z_i^k.$$ 
Recall that $\Sym_K(\zz)$ is generated freely as a commutative $K$-algebra by each of these sets of symmetric functions. In other words, if $(g_k)_{k\geq 1}$ stands for any one of these families, then $(g_\la)_{\la}$ forms a basis of $\Sym_K(\zz)$, where $\la=(\la_1, \ldots,\la_k)$ runs through all integer partitions and $g_\la:=g_{\la_{1}}\cdots g_{\la_k}$.
Lastly, the \emph{duality mapping} $\om\equiv \om_\zz$ is defined as the algebra homomorphism of $\Sym_K(\zz)$ such that $\om(e_k)=h_k$. As is well known, $\om$ also satisfies $\om(h_k)=e_k$ and $\om(p_k)=(-1)^{k-1}p_k$. The following result is~\cite[Thm 4.2]{stanley_symmetric}, and we give an alternative proof.

\begin{prop}[\cite{stanley_symmetric}]\label{prop:stanley-sym}
With the above notation,
$$\om(X_G)(\zz)=\sum_{(\ga,f)}~\prod_{v\in[n]}z_{f(v)},$$
where the sum is over the set $\mC$ of pairs $(\ga,f)$ where $\ga$ is an acyclic orientation of $G$ and $f$ is a $\PP$-coloring without $\ga$-descent.
\end{prop}

\begin{proof}
We claim that 
\begin{equation}\label{eq:dualTH}
\om\left(\prod_{i= 1}^\infty T(z_i\xx)\right)=\prod_{i= 1}^\infty H(z_i\xx).
\end{equation}
Here and in the following we are actually extending $\om$ to the larger space of \emph{symmetric power series} in $\zz$ with coefficients in $K$ (in other words, we allow for symmetric functions of infinite degree), and we can take $K$ to be the field $\QQ(\xx)$ of rational functions in $\xx$ with rational coefficients. 
Observe that for any scalar $t$ in the underlying field $K$,
$$\om\left( \prod_{i=1}^\infty(1+t\, z_i)\right)=\om \left(\sum_{k=0}^\infty e_k t^k\right)=\sum_{k=0}^\infty h_k t^k=\prod_{i=1}^\infty \frac{1}{1-t\,z_i}.$$
Now let $Q(Z)\in K[Z]$ be a polynomial such that $Q(0)=1$. Working in the algebraic closure $\ov K$ of~$K$, one can write $Q(Z)=\prod_{k=1}^d(1+t_kZ)$ with $t_1,\ldots,t_d\in \ov K$. Then, still working over $\ov K$, one gets
$$\om\left( \prod_{i=1}^\infty Q(z_i)\right)=\prod_{k=1}^d\om\left(\prod_{i=1}^\infty(1+t_kz_i)\right)= \prod_{k=1}^d\prod_{i=1}^\infty\frac{1}{1-t_kz_i}=\prod_{i=1}^\infty\frac{1}{Q(-z_i)}.$$
Applying this identity to the polynomial $Q(Z):=T(Z\xx)$ gives~\eqref{eq:dualTH}. Hence,
\begin{eqnarray}
\om(X_G)(\zz)&=&\om\left( [x_1\cdots x_n]\prod_{i= 1}^\infty T(z_i\xx)\right)=[x_1\cdots x_n]\,\om\!\left(\prod_{i= 1}^\infty T(z_i\xx)\right)\nonumber\\
&=&[x_1\cdots x_n]\prod_{i= 1}^\infty H(z_i\xx).\label{eq:dualX}
\end{eqnarray}

Expanding the right-hand side of Equation~\eqref{eq:dualX}, we obtain that $\om(X_G)(\zz)$ is the sum of the monomials  $z_1^{|V_1|}z_2^{|V_2|}\cdots$ over all infinite sequences $((V_1,\ga_1),(V_{2},\ga_{2}),\ldots)$, where $V$ is the disjoint union of the sets $V_i$ and $\gamma_i$ is an acyclic orientation on $G[V_i]$ for all $i\in \PP$. Now Proposition~\ref{prop:stanley-sym} follows using the correspondence detailed after Proposition~\ref{prop:stanley}.
\end{proof}

For an acyclic orientation $\ga$ of $G$ with source-components $S_1,\ldots,S_k$, we denote $\la(\ga)$ the partition of $n$ obtained by ordering the sizes $|S_i|$ in a weakly decreasing manner. 

\begin{prop}\label{prop:greenezaslavsky-sym}
With the above notation,
$$X_G(\zz)=(-1)^n\sum_{\ga \in \mA}(-1)^{\ell(\la(\ga))}p_{\la(\ga)},$$
where the sum is over the set $\mA$ of acyclic orientations of $G$, and $\ell(\la(\ga))$ is the number of source-components of $\ga$.
Equivalently,
\begin{equation}\label{eq:greenezaslavsky-sym}
\om(X_G)(\zz)=\sum_{\ga \in \mA}p_{\la(\ga)}.
\end{equation}
\end{prop}

\begin{example} 
For the graph $G$ in Figure \ref{fig:example-graph-compo}, the chromatic symmetric function is given in Example~\ref{exp:chrom-sym}, and one can compute
\begin{eqnarray*}
X_G(\zz)&=&p_{1,1,1,1}-4p_{2,1,1}+4p_{3,1}+2p_{2,2}-3p_{4};\\
\om(X_G)(\zz)&=&p_{1,1,1,1}+4p_{2,1,1}+4p_{3,1}+2p_{2,2}+3p_{4}.
\end{eqnarray*}
As stated in Theorem \ref{prop:greenezaslavsky-sym}, the coefficients obtained in this expansions correspond to the fact that the number of acyclic orientations $\ga$ of $G$ with partition $\la(\ga)$ equal to $(1,1,1,1)$ (resp. $(2,1,1)$, $(3,1)$, $(2,2)$, $(4)$) is 1 (resp. 4, 4, 2, 3). This matches the direct count one can do by looking at Figure \ref{fig:example-graph-compo}.
\end{example}

\begin{proof} It suffices to prove~\eqref{eq:greenezaslavsky-sym}, since the other identity follows by applying $\om$.
Recall from Theorem~\ref{thm:heap}, that $\ds H(\xx)=\exp(P(\xx)):=\sum_{k=0}^\infty \frac{P(\xx)^k}{k!}$, where $P$ is the generating function of $G$-pyramids. This gives
\begin{eqnarray*}
\prod_{i= 1}^\infty H(z_i\xx)&=&\exp\!\left(\sum_{i=1}^\infty P(z_i\xx)\right)~=~\exp\!\left(\sum_{\he\in\mP} \frac{p_{\she}}{\she}\xx^\he\right)~=~\exp\!\left(\sum_{\mm\in\NN^n}|\mB_\mm|\frac{p_{|\mm|}}{|\mm|}\xx^\mm\right)
.
\end{eqnarray*}
where for $\mm=(m_1,\ldots,m_n)$ we denote $|\mm|=\sum_im_i$, and we let $\mB_{\mm}$ be the set of $G$-pyramids of type $\mm$. Hence
\begin{eqnarray}\label{eq:HP}
\prod_{i= 1}^\infty H(z_i\xx)&=&\prod_{\mm\in\NN^n}\exp\!\left(|\mB_\mm|\,\frac{p_{|\mm|}}{|\mm|}\xx^\mm\right).
\end{eqnarray}
Thus, by~\eqref{eq:dualX}, 
\begin{eqnarray*}
\om(X_G)(\zz)&=&
[x_1\cdots x_n]\prod_{\mm\in\{0,1\}^n}\exp\!\left(|\mB_\mm|\,\frac{p_{|\mm|}}{|\mm|}\xx^\mm\right)\\
&=& [x_1\cdots x_n]\prod_{\mm\in\{0,1\}^n}\left(1+|\mB_\mm|\,\frac{p_{|\mm|}}{|\mm|}\xx^\mm\right).
\end{eqnarray*}
For $V\subseteq [n]$, let $\mB_V$ be the set of acyclic orientations of $G[V]$ with unique source $\min(V)$ (with the convention $\mB_\emptyset=\emptyset$).
By Lemma~\ref{lem:chrom2}, $|\mB_V|=\frac{|\mB_{\mm}|}{|\mm|}$ if $V\neq \emptyset$, where $\mm\in\{0,1\}^n$ is the tuple encoding the set $V$. Hence
$$\om(X_G)(\zz)=[x_1\cdots x_n] \prod_{V\subseteq [n]}\left(1+|\mB_V|\, p_{|V|}\xx^V\right)=\sum_{\{(V_1,\ga_1),\ldots,(V_i,\ga_i)\}}~\prod_{k=1}^ip_{|V_k|},$$
where the sum is over the set $\mB$ of sets of pairs $\{(V_1,\ga_1),\ldots,(V_i,\ga_i)\}$ such that $V_1,\ldots,V_i$ form a set partition of $[n]$, and for all $k\in [i]$ $\ga_k$ is in $\mB_{V_k}$. 
Reasoning as in Section~\ref{sec:specializations}, we can identify $\mB$ with the set of acyclic orientations and the sets $V_i$ with the corresponding source-components. This proves~\eqref{eq:greenezaslavsky-sym}.
\end{proof}

\begin{remark}
Proposition~\ref{prop:greenezaslavsky-sym} could alternatively be obtained by combining~\cite[Theorem 2.6]{stanley_symmetric} with~\cite[Theorem 7.3]{greenezaslavsky}. Indeed,~\cite[Theorem 2.6]{stanley_symmetric} expresses the coefficient of $p_\la$ in $X_G$ in terms of the M\"obius function of the {\em bond lattice} of $G$, and~\cite[Theorem 7.3]{greenezaslavsky} shows that this M\"obius function has the combinatorial interpretation given in Proposition~\ref{prop:greenezaslavsky-sym}.
\end{remark}

As we now explain, Propositions~\ref{prop:stanley-sym} and~\ref{prop:greenezaslavsky-sym} are refinements of Propositions~\ref{prop:stanley} and~\ref{prop:greenezaslavsky} respectively. Let $q$ be an indeterminate, and let $X_G(\zz)_{|\forall k>0,~p_k=q}$ denote the polynomial in $q$ obtained by substituting each of the generators $p_1,p_2,\ldots$ by $q$. 
We observe that 
\begin{equation}\label{eq:evalq}
X_G(\zz)_{|\forall k>0,~p_k=q}=\chi_G(q)
\end{equation}
and for any non-negative integer $j$,
\begin{equation}\label{eq:eval-dual}
\om(X_G)(\1^j)=(-1)^n\chi_G(-j).
\end{equation}
Indeed the polynomials in~\eqref{eq:evalq} coincide on positive integers by~\eqref{eq:X-chi} (since $p_k(\1^j)=j$), and $\om(X_G)(\1^j)=(-1)^n X_G(\zz)_{|\forall k>0,~p_k=-j}$ (since $\om(p_k)=-(-1)^{k}p_k$ and $X_G$ is homogeneous of degree $n$). Thus, specializing Proposition~\ref{prop:stanley-sym} at $\zz=\1^j$ gives Proposition~\ref{prop:stanley}, and specializing Proposition~\ref{prop:greenezaslavsky-sym} at $(p_1,p_2,\ldots)=(q,q,\ldots)$ gives Proposition~\ref{prop:greenezaslavsky}.\\


We now give a refinement of Theorem~\ref{thm:main}. Consider a second set of variables $\yy=(y_1,y_2,\ldots)$.
For a symmetric function $f=f(\zz)$, we denote $f(\yy +\zz)$ the symmetric function in $\yy$ and $\zz$ obtained by substituting the variable $z_{2i-1}$ by $y_i$ and $z_{2i}$ by $z_i$ for all $i\in \PP$ (equivalently, substituting the generator $p_i=p_i(\zz)$ by $p_i(\yy)+p_i(\zz)$).

\begin{thm}\label{thm:main-symmetric}
Let $G$ be a graph. Let $\mD$ be the set of pairs $(\ga,f)$, where $\ga$ is an acyclic orientation of $G$ and $f:V\to \NN$ is an $\NN$-coloring of $G$ without $\ga$-descent. Then
$$\om(X_G)(\yy+ \zz)=\sum_{(\ga,f)\in \mD}\,p_{\la(\ga_0)}(\yy)\,\prod_{i\in \PP}z_i^{|f^{-1}(i)|}$$
where $\ga_0$ is the restriction of $\ga$ to $G[f^{-1}(0)]$.
\end{thm}

Observe that Corollary~\ref{cor:equivThm1} (which is equivalent to Theorem~\ref{thm:main}) is the specialization of Theorem~\ref{thm:main-symmetric} obtained by substituting $p_k(\yy)$ by $q$ and $p_k(\zz)$ by $j$ for all $k\in\PP$, and then taking the coefficient of $q^i$. Observe also that setting $\yy=0$ in Theorem~\ref{thm:main-symmetric} gives Proposition~\ref{prop:stanley-sym}, while setting $\zz=0$ gives Proposition~\ref{prop:greenezaslavsky-sym}.

\begin{proof}
By~\eqref{eq:dualX},
\begin{eqnarray*}
\om(X_G)(\yy+\zz)&=&[x_1\cdots x_n]\left(\prod_{i= 1}^\infty H(y_i\xx)\right)\cdot \left(\prod_{i= 1}^\infty H(z_i\xx)\right),\\
&=&\sum_{U\uplus V=[n]}\left([\xx^U]\prod_{i= 1}^\infty H(y_i\xx)\right)\cdot \left([\xx^V]\prod_{i= 1}^\infty H(z_i\xx)\right).
\end{eqnarray*}
where the sum is over the pairs $(U,V)$ of disjoint sets whose union is $[n]$.
Applying  \eqref{eq:dualX} to the induced graphs $G[U]$ and $G[V]$ gives 
\begin{eqnarray*}
\om(X_G)(\yy+\zz)&=&\sum_{U\uplus V=[n]}\om(X_{G[U]})(\yy)\cdot \om(X_{G[V]})(\zz).
\end{eqnarray*}
Lastly, applying Propositions~\ref{prop:greenezaslavsky-sym} and~\ref{prop:stanley-sym} to $\om(X_{G[U]})(\yy)$ and $\om(X_{G[V]})(\zz)$ respectively gives 
\begin{eqnarray*}
\om(X_G)(\yy+\zz)&=&\sum_{U\uplus V=[n]}\left(\sum_{\ga_0\in \mA(U)}p_{\la(\ga_0)}(\yy)\right)\cdot\left(\sum_{(\ga',f')\in\mC(V)}~\prod_{v\in V} z_{f'(v)}\right),
\end{eqnarray*}
where $\mA(U)$ is the set of acyclic orientations of $G[U]$, and $\mC(V)$ is the set of pairs $(\ga',f')$ with $\ga'$ acyclic orientation of $G[V]$ and $f'$ a $\PP$-coloring of $G[V]$ without $\ga'$-descent. Theorem~\ref{thm:main-symmetric} follows by identifying $ \bigcup_{U\uplus V=[n]}\mA(U)\times \mC(V)$ with $\mD$ (identifying $U$ with the set $\ga^{-1}(0)$ of vertices colored 0, etc.).
\end{proof}

As the proof of Theorem~\ref{thm:main-symmetric} shows, it is easy to combine several results into one, at the cost of using several sets of variables. This is because our identities hold at the level of the heap generating function $\prod_{i= 1}^\infty T(z_i\xx)$. For instance, it is straightforward to recover the \emph{superfication} result~\cite[Thm 4.3]{stanley_symmetric}, as we now explain. 

We denote by $X_G(\yy-\zz)$ the function of $\yy$ and $\zz$   obtained from $X_G(\zz)$ by substituting $p_k(\zz)$ by $p_k(\yy)-(-1)^{k}p_k(\zz)$. Equivalently, $X_G(\yy-\zz)$ is  obtained from $X_G(\yy+\zz)$ by applying duality \emph{only on the $\zz$ variables}:
$$X_G(\yy-\zz):=\om_\zz(X_G(\yy+\zz)).$$
Using~\eqref{eq:sym-heap} and~\eqref{eq:dualTH} gives
\begin{eqnarray*}
X_G(\yy-\zz)=[x_1\cdots x_n] \left(\prod_{i= 1}^\infty T(y_i\xx) \right)\cdot \left( \prod_{i= 1}^\infty H(z_i\xx)\right)
=\sum_{U\uplus V=[n]}X_{G[U]}(\yy)\cdot \om(X_{G[V]})(\zz)
\end{eqnarray*}
Hence, 
$$X_G(\yy-\zz)=\sum_{U\uplus V=[n]}\sum_{(f_-,f_+,\ga_+)}y_i^{|f_-^{-1}(i)|}\, z_i^{|f_+^{-1}(i)|},$$
where the inner sum is over the set of triples $(f_-,f_+,\ga^+)$ such that $f_-$ is a proper $\PP$-coloring of $G[U]$, $\ga_+$ is an acyclic orientation of $G[V]$, and $f_+$ is a $\PP$-coloring of $G[V]$ without $\ga_+$-descent. Equivalently (upon coloring $U$ with negative colors, and extending $\ga_+$ to $G$), one gets
\begin{equation}\label{eq:superfication}
X_G(\yy-\zz)=\sum_{(\ga,f)}\,\prod_{i\in\PP} y_i^{|f^{-1}(-i)|}\, z_i^{|f^{-1}(i)|},
\end{equation}
where the sum is over pairs $(\ga,f)$ where $\ga$ is an acyclic orientation of $G$ and $f:V\to \ZZ\setminus\{0\}$ is a coloring without $\ga$-descent such that for all $i<0$ the vertices of color $i$ are pairwise non-adjacent. This is exactly~\cite[Thm 4.3]{stanley_symmetric}.

There is no obstacle to pursuing this idea further. For instance, one can combine~\eqref{eq:superfication} and Theorem~\ref{thm:main-symmetric} into a single statement. Consider a new set of variables $\zz'=(z_1',z_2'\ldots)$, and the function $X_G(\yy-(\zz+\zz'))$ obtained from $X_G(\zz)$ by substituting $p_k(\zz)$ by $p_k(\yy)-(-1)^{k}(p_k(\zz)+p_k(\zz'))$. Let $\mE$ be the set of pairs $(\ga,f)$, where $\ga$ is an acyclic orientation of $G$ and $f:V\to \ZZ$ is an $\ZZ$-coloring of $G$ without $\ga$-descent, such that for all $i<0$ the vertices of color $i$ are pairwise non-adjacent. Then
\begin{eqnarray}\label{eq:combine-sym}
X_G(\yy-(\zz+\zz'))=\sum_{(\ga,f)\in \mE}\,p_{\la(\ga_0)}(\zz')\, \prod_{i\in\PP} y_i^{|f^{-1}(-i)|} \,z_i^{|f^{-1}(i)|} ,
\end{eqnarray}
where $\ga_0$ is the restriction of $\ga$ to $G[f^{-1}(0)]$. 
Note that setting $\yy=0$ in~\eqref{eq:combine-sym} gives Theorem~\ref{thm:main-symmetric}, while setting $\zz'=0$ gives~\eqref{eq:superfication}.

\begin{remark}
Recall the notion of proper multicolorings from Remark~\ref{rem:multicolorings}. For $\mm\in \NN$, the symmetric function 
\begin{equation}
\label{eq:X_m}
X_{G,\mm}(\zz):=[\xx^\mm]\prod_{i= 1}^\infty T(z_i\xx),
\end{equation}
can be interpreted as counting proper multicolorings of $G$ of type $\mm$ according to the number of times each color in $\PP$ is used. By the same reasoning as in Remark~\ref{rem:multicolorings}, one gets
\[X_{G,\mm}(\zz)=\frac{X_{G^\mm}(\zz)}{\mm!}.\]
so that these generalized chromatic symmetric functions are still chromatic symmetric functions, up to a multiplicative constant. Hence the results in this section apply to $X_{G,\mm}$. 
This was noticed already in~\cite[Eq. (3)]{stanley_graphcoloring}. In fact~\cite[Proposition 2.1]{stanley_graphcoloring} follows from the combinatorial interpretation of~\eqref{eq:X_m}.
\end{remark}

\noindent \textbf{Acknowledgment.} We thank the anonymous referees for their numerous careful comments.

\bibliographystyle{plain} 
\bibliography{biblio-chromatic} 

\begin{thebibliography}{10}

\bibitem{cartierfoata}
Pierre Cartier and Dominique Foata.
\newblock {\em Probl\`emes combinatoires de commutation et r\'{e}arrangements}.
\newblock Lecture Notes in Mathematics, No. 85. Springer-Verlag, Berlin-New
  York, 1969.

\bibitem{Crapo:bipolar-orient}
Henry~H. Crapo.
\newblock A higher invariant for matroids.
\newblock {\em Journal of Combinatorial Theory}, 2(4):406--417, 1967.

\bibitem{Deb:chromatic-heap}
Bishal Deb.
\newblock Chromatic polynomial and heaps of pieces.
\newblock arXiv preprint arXiv:1902.02240., 2019.

\bibitem{bivariate}
Klaus Dohmen, Andr\'{e} P\"{o}nitz, and Peter Tittmann.
\newblock A new two-variable generalization of the chromatic polynomial.
\newblock {\em Discrete Math. Theor. Comput. Sci.}, 6(1):69--89, 2003.

\bibitem{gasharov}
Vesselin Gasharov.
\newblock Incomparability graphs of {$(3+1)$}-free posets are {$s$}-positive.
\newblock In {\em Proceedings of the 6th {C}onference on {F}ormal {P}ower
  {S}eries and {A}lgebraic {C}ombinatorics ({N}ew {B}runswick, {NJ}, 1994)},
  volume 157, pages 193--197, 1996.

\bibitem{gessel:chromatic-heaps}
Ira~M. Gessel.
\newblock Acyclic orientations and chromatic generating functions.
\newblock {\em Discrete Math.}, 232(1-3):119--130, 2001.

\bibitem{greenezaslavsky}
Curtis Greene and Thomas Zaslavsky.
\newblock On the interpretation of {W}hitney numbers through arrangements of
  hyperplanes, zonotopes, non-{R}adon partitions, and orientations of graphs.
\newblock {\em Trans. Amer. Math. Soc.}, 280(1):97--126, 1983.

\bibitem{Krattenthaler:Heaps}
Christian Krattenthaler.
\newblock The theory of heaps and the {C}artier-{F}oata monoid.
\newblock Appendix of the electronic edition of \emph{Probl\`emes combinatoires
  de commutation et r\'earrangements}., 2006.

\bibitem{lass:chromatic-heaps}
Bodo Lass.
\newblock Orientations acycliques et le polyn\^{o}me chromatique.
\newblock {\em European J. Combin.}, 22(8):1101--1123, 2001.

\bibitem{stanley_acyclic}
Richard~P. Stanley.
\newblock Acyclic orientations of graphs.
\newblock {\em Discrete Math.}, 5:171--178, 1973.

\bibitem{stanley_symmetric}
Richard~P. Stanley.
\newblock A symmetric function generalization of the chromatic polynomial of a
  graph.
\newblock {\em Adv. Math.}, 111(1):166--194, 1995.

\bibitem{stanley_graphcoloring}
Richard~P. Stanley.
\newblock Graph colorings and related symmetric functions: ideas and
  applications: a description of results, interesting applications, \& notable
  open problems.
\newblock {\em Discrete Math.}, 193(1-3):267--286, 1998.
\newblock Selected papers in honor of Adriano Garsia (Taormina, 1994).

\bibitem{viennot:heaps_newyork}
G\'{e}rard~X. Viennot.
\newblock Heaps of pieces. {I}. {B}asic definitions and combinatorial lemmas.
\newblock In {\em Graph theory and its applications: {E}ast and {W}est
  ({J}inan, 1986)}, volume 576 of {\em Ann. New York Acad. Sci.}, pages
  542--570. New York Acad. Sci., New York, 1989.

\end{thebibliography}

\end{document}